\documentclass[11pt]{article}
\usepackage{amsmath,amsthm,amsfonts,amssymb,amscd,amsxtra,mathrsfs,commath,url}
\usepackage{graphicx,float}
\usepackage{url}
\usepackage[margin=2.7 cm,nohead]{geometry}
\usepackage{bbm}
\usepackage{cite}
\usepackage{threeparttablex,color,soul,colortbl,hhline,rotating,booktabs,longtable}
\usepackage[table]{xcolor}
\definecolor{lightgray}{gray}{0.95}
\newtheorem{theorem}{Theorem}
\newtheorem{lemma}[theorem]{Lemma}
\newtheorem{definition}{Definition}
\newtheorem{corollary}[theorem]{Corollary}
\newtheorem{proposition}[theorem]{Proposition}

\newtheorem*{remark}{Remark}

\newtheorem{algorithm}{Algorithm}
\newcommand{\R}{\mbox{${I\!\!R}$}}

\newcommand{{\coes}}{\c c\~oes }

\newcommand{\be}{\begin{equation}}
\newcommand{\ee}{\end{equation}}
\newcommand{\beqn}{\begin{eqnarray}}        
\newcommand{\eeqn}{\end{eqnarray}}

\newcommand{\la}{\langle}
\newcommand{\ra}{\rangle}
\DeclareMathOperator{\grad}{grad}

\begin{document}
\title{An inexact proximal point method for variational \\ inequality  on Hadamard manifolds}
\author{
G. C. Bento   \thanks{Instituto de Matem\'atica e Estat\'istica, Universidade Federal de Goi\'as,  CEP 74001-970 - Goi\^ania, GO, Brazil, E-mails: {\tt  glaydston@ufg.br},  {\tt  orizon@ufg.br}. The authors was supported in part by  CNPq grants 423737/2016-3, 310864/2017-8, 305158/2014-7, 408151/2016-1 and 302473/2017-3,  FAPEG/PRONEM- 201710267000532.}
\and
O.  P. Ferreira  \footnotemark[1]
\and
E.A. Papa Quiroz     \thanks{Universidad Nacional Mayor de San Marcos and Universidad Privada del Norte, Lima, Peru,  {\tt   erikpapa@gmail.com}}
}
\maketitle

%%%%%%%%%%%%%%%%%%%%%%%%%%%%%%%%%%%%%%%%%%%%%%%%%%%%%%%%
\begin{abstract}
In this paper we present an inexact proximal  point method for variational inequality problem  on Hadamard manifolds and study its convergence properties.   The proposed  algorithm is inexact in two sense. First, each  proximal subproblem is  approximated  by using the  enlargement of the  vector  field in consideration and then  the next iterated is obtained by solving this subproblem allowing a suitable  error tolerance. As an application, we obtain an inexact proximal  point method for constrained optimization problems, equilibrium problems  and nonlinear optimization problems on Hadamard manifolds.\\

\end{abstract}
\section{Introduction}
%%%%%%%%%%%%%%%%%%%%%%%%%%%%%%%%%%%%%%%%%%%%%%%%%%%%%%%%
Extensions of concepts and techniques of optimization from the Euclidean  space to the Riemannian context have been a subject of intense research in recent years. An special attention  has been given to methods of  Riemannian mathematical programming; papers published on  this topic involving proximal point methods include, but are not limited to, \cite{AhmadiKhatibzadeh2014, BentoNetoOliveira2016, Bacak2013, BentoFerreira2015, LiYao2012, PapaQuirozOliveira2012, TangHuang2013, WangLiLopezYao2015,WangLiChongGenaro2016}.  It is well known that  one of the reasons for this extension is the possibility of transforming non-convex or non-monotone problems in the Euclidean context into  Riemannian convex or monotone  problems, by introducing a suitable metric, which  enables modified  numerical  methods to find  solutions for these problems;  see \cite{BentoFerreira2015, BentoMelo2012,  FerreiraLouzeiroPrudente2019, CLMM2012, FerreiraCPN2006, Rapcsak1997}. Moreover,  constrained optimization problems can be viewed as unconstrained ones from a Riemannian geometry point of view. In particular, many Euclidean optimization problems are naturally posed on the Riemannian context;  see, e.g., \cite{AdlerDedieuShub2002, EdelmanAriasSmith1999, JeurisVandebrilVandereycken2012,  Karmarkar1998,  Luenberger1972,  NesterovTodd2002,  Smith1994, SraHosseini2015,  Udriste1994,  ZhangReddiSra2016, Xu2019}.

In this paper, we consider the problem of  finding a  solution of a variational inequality  problem defined on  Riemannian manifolds. Variational inequality on Riemannian manifolds were first introduced and studied  by  N\'emeth in \cite{Nemeth2003},  for univalued vector fields on Hadamard manifolds,   and  for multivalued  vector fields on general Riemannian manifolds  by Li and Yao in   \cite{LiYao2012};  for recent works addressing this subject  see \cite{FangChen2015, LiChongLiouYao2009, TangWangLiu2015, TangZhouHuang2013}. It is worth noting that  constrained optimization problems  and the problem of finding the zero of a  multivalued vector field on Riemannian manifolds,  which were  studied  in \cite{AhmadiKhatibzadeh2014, BentoFerreira2015, daCruzFerreiraPerez2006,  FerreiraOliveira2002, LiLopesMartin-Marquez2009, WangLiLopezYao2015},  are  particular instances of the variational inequality problem. 

The aim  of this   paper is to  introduce  an inexact proximal  point method for variational inequality problem  in Hadamard manifolds and to study its convergence properties. The proposed  algorithm combine ideas from the papers \cite{Batista2015} and \cite{WangLiLopezYao2015} to obtain in inexact algorithm  in two sense.  First, each  proximal subproblem is  approximated  by using the  enlargement of the  vector  field in consideration and then the next iterated is obtained by solving this subproblem allowing a suitable  error tolerance.  
This algorithm has as particular instances  some  algorithms  previously  studied.  For instance, it generalize the algorithm studied in  \cite{rocka01} to Riemannian setting,  by considering two of the four errors considered there.   Considering that the Riemannian algorithm studied in \cite{WangLiLopezYao2015, WangLiChongGenaro2016}  does not use the enlargement of the vector field in consideration, then in this sense our algorithm has it as particular instance. Moreover,  our algorithm also  merges into  algorithms studied  in  \cite{BatistaGlaydstonFerreira2016,TangHuang2013}.  It is worth highlighting that the use of enlargement  in the proximal subproblem to define the next iteration of the algorithm has the advantage of providing more latitude and more robustness to the algorithm, as explained in \cite{BurachikIusemSvaiter1997}. The concept of enlargement of monotone operators  in linear spaces has been successfully employed for a wide range of purposes; see \cite{BurachikIusem2008} and its reference therein.   The extension of this concept to Riemannian context  has been presented in \cite{BatistaGlaydstonFerreira2016}. 
As an application, from the our iterative scheme we obtain an inexact proximal  point method for constrained optimization problems, equilibrium problems  and nonlinear optimization problems on Hadamard manifolds.  To the best of our knowledge, our approach brings a first proposal of an inexact proximal method for equilibrium problems on Hadamard manifolds.

It is important to note that an exact version was first introduced in \cite{CLMM2012} and, by using the theory of variational inequality,  has been reaffirmed for genuine  Hadamard manifolds in  \cite{Li2019}.

The  organization of the paper is as follows. In Section~\ref{sec:prelim},   notations   basic results used  thought  the paper are presented. In Section \ref{Sec:InProxVIP}, the inexact proximal point method for variational inequalities is presented  and its convergence properties are studied. As an application, in Section~\ref{sec:appl}, an inexact proximal point method for  constrained optimization problems, equilibrium problems  and nonlinear optimization problems  are obtained. In Section~\ref{secfr} concluding remarks are presented.

%%%%%%%%%%%%%%%%%%%%%%%%%%%%
\section{Preliminaries}  \label{sec:prelim}
%%%%%%%%%%%%%%%%%%%%%%%%%%%%
The aim of the section is to recall  some fundamental  properties and notations  of  Riemannian geometry, as well as  the notions of   monotonicity and   maximal  monotonicity and  enlargement of  multivalued  vector  fields  on Hadamard manifolds;   for more details see  \cite{Batista2015}.

%%%%%%%%%%%%%%%%%%%%%%%%%%%%
\subsection{Notation and terminology} \label{sec:aux}
%%%%%%%%%%%%%%%%%%%%%%%%%%%%
 {\it In this paper, all manifolds $M$ are assumed to be Hadamard finite dimensional}. Next we recall  a fundamental inequality  of Hadamard manifolds  that we will  need 
 \begin{equation} \label{eq:coslaw}
d^2(p_1,p_3)+d^2(p_3,p_2)-2\left\langle \exp_{p_3}^{-1}p_1,\exp_{p_3}^{-1}p_2\right\rangle\leq d^2(p_1,p_2),  \qquad p_1, p_2 , p_3 \in M, 
\end{equation}
 where  $\exp_{p}(\cdot) $ denotes the  {\it exponential map}, $\exp^{-1}_{q}(\cdot)$ its  inverse and $d(\cdot, \cdot)$ is the Riemannian distance. The function $f:M\to\mathbb{R}\cup\{+\infty\}$ is said to be proper if $\mbox{dom}f:=\{p\in M:~f(p)<+\infty\}\neq \varnothing$ and it is {\it convex} on a convex set $\Omega\subset \mbox{dom}\,~f$ if for any geodesic segment $\gamma$ in $\Omega$, the composition $f\circ\gamma$ is convex. It is well known that  $d^2(q, \cdot)$ is convex. The  {\it subdifferential\/} of $f$ at $p$ de defined by  $ \partial f(p)=\{f(q) \geq f(p) + \langle s, \, \exp^{-1}_pq \rangle, ~q\in M\}$.  The function  $f$ is {\it lower semicontinuous} at $\bar{p}\in\mbox{dom}f$ if for each sequence $\{p^k\}$ converging to $\bar{p}$, we have $\liminf_{k\rightarrow\infty} f(p^k)\geq f(\bar{p}).$
 Denotes by  $X: M \rightrightarrows  TM$ with $X(p)\subset T_{p}M$ a multivalued vector field  and by
$\mbox{dom}X:=\left\{ p\in M ~: ~X(p)\neq \varnothing \right\},$ its    domain.  We say that  $X$ is {\it bounded on bounded sets }  if for all bounded set $V\subset M$  such that its closure $\overline{V} \subset  \mbox{int}( \mbox{dom}\,X)$ it holds that    $m_X(V):=\sup_{q\in V}\left\{\|u\|~:~u\in X(q)\right\}< +\infty$; see an equivalent definition in \cite{LiLopesMartin-Marquez2009}.  For two  multivalued vector fields  $X, Y$ on $M$, the notation $ X\subset Y$ implies that $X(p)\subset Y(p)$, for all $p\in M$.  Denotes by {\it  $P_{pq}$ the parallel transport along the geodesic  from $p$ to $q$}. A multivalued vector field $X$ satisfying  $\langle P_{qp}^{-1} u-v, \, \exp_{q}^{-1}p\rangle \geq 0$ and $\langle P_{qp}^{-1} u-v, \,\exp_{q}^{-1}p\rangle \geq \rho d^2(p,q)$,  for some $\rho>0$  and all $p,\,q\in \mbox{dom}\,X$ and $u\in X(p), ~ v\in X(q)$,  is said to be  {\it monotone}, respectively,   {\it strongly monotone}.  Moreover,  a monotone vector field $X$  is said to be  {\it maximal monotone\/}, if for each $p\in \mbox{dom}\,X$ and  all $u\in T_pM$, there holds: $$\langle P_{qp}^{-1} u-v, \,\exp_{q}^{-1}p\rangle \geq 0,  q\in \mbox{dom}\,X, \quad ~ v\in X(q)  ~ \Rightarrow ~  u\in X(p).$$
 For more details about monotonicity  of vector field; see  \cite{Nemeth1999, NetoFerreiraLucambio2000, LiLopesMartin-Marquez2009}. The proof of the  next result   can be found in   \cite[Theorem~5.1]{LiLopesMartin-Marquez2009}.
\begin{theorem}\label{mmsub}
 Let $f$ be a proper, lower semicontinuous and convex function on $M$. The subdifferential $\partial f$ is a monotone multivalued vector field. Furthermore, if $\emph{dom}\,f=M$, then the subdifferential $\partial f$ of $f$ is a maximal monotone vector field.
\end{theorem}
 Let $\Omega\subset\mathbb{R}^n$ be a convex set, and $p\in \Omega$. From \cite{LiLopesMartin-Marquez2009}, we define the {\it normal cone} to $\Omega$ at $p$ by
\begin{equation} \label{eq:nc}
N_{\Omega}(p):=\left\{w\in T_pM~:~\left\langle w, \exp_{p}^{-1}q \right\rangle\leq 0, ~q\in \Omega \right\}.
\end{equation}
The   {\it  indicator function}  $\delta_{\Omega}:M\to\mathbb{R}\cup\{+\infty\}$ of  the set $\Omega $  is defined by  $\delta_{\Omega}(p)=0$, for  $p\in\Omega$ and $\delta_{\Omega}(p)=+\infty$ otherwise. The next result can be found in \cite[Proposition 5.4]{LiLopesMartin-Marquez2009}.
\begin{proposition}  \label{pr:pif}
Let  $\Omega\subset M$ be a closed and convex set and $f:M\to\mathbb{R}$ be a convex function. Then, $\partial \delta_{\Omega}(p)=N_{\Omega}(p)$ and  $ \partial (f+\delta_{\Omega})(p)= \partial f(p)+ N_{\Omega}(p)$,  for all $p\in \Omega$. 
\end{proposition} 
 The proof of the next result follows from   \cite[ Corollary 3.14]{LiYao2012}. 
\begin{lemma} \label{le:esvi}
Let $X$ be a maximal monotone vector field such that $\emph{dom}\,X=M$.   For each $q\in M$ and $\lambda > 0$,  the inclusion problem
$
0\in X(p) +N_{\Omega}(p)-  \lambda  \exp^{-1}_p q, 
$
for $p\in M$, has an unique solution.
\end{lemma}
Since the exponential mapping is continuous in both arguments, the next proposition is an immediate consequence of definition \eqref{eq:nc}, for that its proof wil be omite. 
\begin{proposition} \label{eq:nccc}
Let   $C\subset M$  be a closed set.  If $\overline{p}=\lim_{k\rightarrow\infty}p^k$, $\overline{u}=\lim_{k\rightarrow\infty}u^k$, and $u^k\in N_{\Omega}(p^k)$ for all $k$, then $\overline{u}\in N_{\Omega}(\overline{p})$. 
\end{proposition}

We end this section with a real analysis result,  see the  proof in  \cite[Lemma~2, pp. 44]{polya}.
\begin {lemma}\label{le:qfm}
Let $ \left \{  \zeta_{k}\right \}$,$ \left \{  \gamma _{k}\right \}$, $\left \{  \beta  _{k}\right \}$ be  sequences of  nonnegative real numbers satisfying $ \sum_{k=1}^{\infty }\gamma  _{k}< \infty$ and 
$ \sum_{k=1}^{\infty }\beta _{k}<\infty$.   If   $\zeta_{k+1}\leq \left ( 1+ \gamma _{k}\right )\zeta_{k}+\beta _{k}$,  then  $\left \{ \zeta_{k} \right \} $ converges.
\end {lemma}
%%%%%%%%%%%%%%%%%%%%%%%%%%%%%%%%%%%%%%
\subsection{Enlargement of Monotone Vector Fields} \label{sec:EnlMon}
%%%%%%%%%%%%%%%%%%%%%%%%%%%%%%%%%%%%%%
 In this section  we  recall some concepts and results related to  enlargement of vector fields  in  the   Hadamard manifolds  setting, for details see \cite{BatistaGlaydstonFerreira2016}. Throughout this section    $X$  and $Y$  denote multivalued monotone vector fields on $M$   and $\epsilon \geq 0$.
\begin{definition} \label{def.enl.X}
The enlargement of  vector field  $X^{\epsilon}: M   \rightrightarrows  TM $  associated to  $X$   is defined by
$$
X^{\epsilon}(p):=\left\{ u\in T_pM~:~ \left\langle P_{qp}^{-1} u-v, \,\exp_{q}^{-1}p\right\rangle \geq  -\epsilon, ~  q\in \mbox{dom}\,X, ~  v\in X(q) \right\}, \quad  p\in \mbox{dom}\,X.
$$
\end{definition}
Next proposition  shows  that $X^{\epsilon}$ effectively constitutes  an enlargement to $X$.
\begin{proposition} \label{prop.elem.ii} 
  $X\subset X^{\epsilon}$ and   $\emph{dom}\,X \subset \emph{dom}\,X^\epsilon$. In particular,  if $\emph{dom}\, X=M$ then $\emph{dom}\, X^\epsilon=\emph{dom}\,X$. Moreover, if $X$  is maximal then    $X^0=X$.
\end{proposition}
In the next three  propositions we state the main properties used throughout  our presentation, which are  extensions to  the Riemannian context of the corresponding one of linear setting; see \cite{BurachikIusemSvaiter1997}.
\begin{proposition} \label{prop.elem.X}
$X^{\epsilon_2}\subset X^{\epsilon_1}$, for all $\epsilon_1\geq\epsilon_2\geq0$, and $X^{\epsilon_1}+Y^{\epsilon_2}\subset(X+Y)^{\epsilon_1+\epsilon_2}$.
\end{proposition}
\begin{proposition} \label{prop.conv.alg.}
Let  $\{ \epsilon^k\}$ be a sequence of positive numbers, and $\{ (p^k, \, u^k)\} $ be a sequence in $TM$.  If $\overline{\epsilon}=\lim_{k\rightarrow\infty}\epsilon^k$, $\overline{p}=\lim_{k\rightarrow\infty}p^k$, $\overline{u}=\lim_{k\rightarrow\infty}u^k$, and $u^k\in X^{\epsilon_k}(p^k)$ for all $k$, then $\overline{u}\in X^{\overline{\epsilon}}(\overline{p})$;
\end{proposition}
\begin{proposition} \label{prop.boun.boun.}
If $X$ is maximal monotone and $\emph{dom}\,X=M$,  then $X^\epsilon$ is bounded on bounded sets,  for all $\epsilon\geq0$.
\end{proposition}
%%%%%%%%%%%%%%%%%%%%%%%%%%%%%%%%%%%%%%
\section{ Inexact Proximal Point Method for Variational Inequalities} \label{Sec:InProxVIP}
%%%%%%%%%%%%%%%%%%%%%%%%%%%%%%%%%%%%%%
In this section,  we introduce  an  inexact version of the proximal  point method for variational inequalities in Hadamard manifolds. It is worth noting that, the variational inequality problem   was first  introduced in \cite{Nemeth2003},  for  single-valued  vector  fields  on Hadamard manifolds,   and   in \cite{LiYao2012} for  multivalued  vector  fields in  Riemannian manifolds. 

Let $X:M\rightrightarrows TM$ be a multivalued vector field and $ \Omega \subset M$ be a nonempty set. The {\it variational inequality problem}  for  $X$ and   $C$,  denoted by VIP(X,$\Omega$), consists  of  finding $p^*\in\Omega$ such that there exists $u\in X(p^*)$ satisfying
\begin{equation} \label{eq:vipp}
 \langle u,\,\exp^{-1}_{p^*}q\rangle\geq0, \qquad   q\in\Omega.
\end{equation}
Using \eqref{eq:nc}, i.e., the definition of normal cone to $\Omega$,  VIP(X,$\Omega$) becomes the problem of finding an $p^*\in\Omega$ that satisfies the inclusion
\begin{equation} \label{eq.vip}
0\in X(p)+N_{\Omega}(p).
\end{equation}
\begin{remark}
In particular, if  $\Omega=M$, then  $N_{\Omega}(p)=\{0\}$ and  \mbox{VIP}\,(X,$\Omega$) are problems with regard to finding $p^*\in\Omega$ such that $0\in X(p^*).$
\end{remark}
Hereafter, $S(X,\,\Omega)$  denotes the  {\it solution set of the inclusion}  (\ref{eq.vip}). We require the following  three assumptions:
\begin{itemize}
\item[{\bf A1.}] $\mbox{dom}\, X=M$ and $\Omega$ closed and convex;
\item[ {\bf A2.}] $X$ is maximal monotone;
\item[ {\bf A3.}] $S(X,\,\Omega)\neq \varnothing$.
\end{itemize} 
In the following we state two  algorithms  to solve \eqref{eq:vipp} or equivalently  \eqref{eq.vip}. 
{\it To state the  algorithms  take  two  real numbers $ \hat{\lambda}$ and  $\tilde{\lambda}$  satisfying      $0<\hat{\lambda}\leq\tilde{\lambda}$ and  four  exogenous  sequences   of positive real  numbers}, $\{\lambda_k\}$,    $\{\sigma_k\}$,  $\{\theta_k\}$ and $\{\epsilon_k\}$ satisfying  
\begin{equation} \label{eq:es}
\hat{\lambda}\leq \lambda_k\leq\tilde{\lambda},\qquad  \qquad \sum_{k=1}^{+\infty} \epsilon_k< + \infty,  \qquad    \qquad  \sum_{k=1}^{+\infty} \sigma_k< + \infty , \qquad    \qquad  \sum_{k=1}^{+\infty} \theta_k< + \infty .
\end{equation}

 \vspace{.6cm}
 
 The first version of {\it inexact proximal point method} for solving \eqref{eq.vip}  is defined as follows: \\
 
\vspace{.2cm}
\hrule
\begin{algorithm} \label{Alg:InexProxae}
 {\bf Inexact proximal point method with absolute error tolerance}  \\
\hrule
\begin{description}
\item[0.] Take $\{\lambda_k\}$,   $\{\epsilon_k\}$ and $\{\theta_k\}$ satisfying \eqref{eq:es}, and    $p^0 \in  \Omega$. Set $k=0$.
\item [1.] Given $p^{k} \in \Omega$, compute   $p^{k+1} \in \Omega$  and  $e^{k+1}\in T_{p^{k+1}}M$ such that
\begin{equation} \label{eq:ineqae}
 e^{k+1} \in X^{\epsilon_k}(p^{k+1})  +N_{\Omega}(p^{k+1})- \lambda_{k} \exp^{-1}_{p^{k+1}}p^{k},
 \end{equation} 
\begin{equation} \label{eq;ecae}
 \| e^{k+1}\| \leq \theta_ {k}.
\end{equation}
\item[ 2.] If  $p^{k}=p^{k+1}$, then {\bf stop}; otherwise,   set $k\gets k+1$, and go to step~{\bf  1}.
\end{description}
\hrule
\end{algorithm}
 \vspace{.2cm}

It is worth to noting that Algorithm~\ref{Alg:InexProxae}  is inexact in two sense, namely, $X^{\epsilon_k}$ is an enlargement of the  vector  field $X$  and   each iteration $p^{k+1}$  is an approximated solution of the vectorial inclusion   $0\in X^{\epsilon_k}(p)  +N_{\Omega}(p)- \lambda_{k} \exp^{-1}_{p}p^{k}$ satisfying the error  criterion  \eqref{eq;ecae}. Note that for $\theta_ {k}=0$ in \eqref{eq;ecae}, Algorithm~\ref{Alg:InexProxae} merges into  algorithm introduced in \cite{BatistaGlaydstonFerreira2016}.  The error criterion  \eqref{eq;ecae} was introduced in  the  celebrated paper \cite{rocka01}  to  analyze an  inexact version  of the  proximal point method to find  zeroes of maximal monotone operators  in  linear context; see  \cite{WangLiLopezYao2015, WangLiChongGenaro2016} for a generalization to Riemannian setting. 
 
The second version of {\it inexact proximal point method} for solving \eqref{eq.vip}  is defined as follows: \\

 \vspace{.2cm}
\hrule
\begin{algorithm} \label{Alg:InexProx}
 {\bf Inexact proximal point method with relative error  tolerance}  \\
\hrule
\begin{description}
\item[0.] Take $\{\lambda_k\}$,   $\{\epsilon_k\}$ and $\{\sigma_k\}$ satisfying \eqref{eq:es}, and    $p^0 \in  \Omega$. Set $k=0$.
\item [1.] Given $p^{k} \in \Omega$, compute   $p^{k+1} \in \Omega$  and  $e^{k+1}\in T_{p^{k+1}}M$ such that
\begin{equation} \label{eq:ineq}
 e^{k+1} \in X^{\epsilon_k}(p^{k+1})  +N_{\Omega}(p^{k+1})- \lambda_{k} \exp^{-1}_{p^{k+1}}p^{k},
 \end{equation}
\begin{equation} \label{eq;ec}
 \| e^{k+1}\| \leq \sigma_ {k}d(p^{k}, p^{k+1}).
\end{equation}
\item[ 2.] If  $p^{k}=p^{k+1}$, then {\bf stop}; otherwise,   set $k\gets k+1$, and go to step~{\bf  1}.
\end{description}
\hrule
\end{algorithm}
 \vspace{.2cm}

First we remark  that Algorithm~\ref{Alg:InexProx} differs from Algorithm~\ref{Alg:InexProxae}  only in the errors  criterion adopted, more precisely,   between  \eqref{eq;ecae} and \eqref{eq;ec}.  The error criterion  \eqref{eq;ec} was also introduced in  \cite{rocka01}    in  linear setting. When $\sigma_{k+1}=0$ in \eqref{eq;ec}, Algorithm~\ref{Alg:InexProx} merges into  algorithm introduced in \cite{BatistaGlaydstonFerreira2016}.  A variant of  error criterium   \eqref{eq;ec}  to analyze \eqref{eq:ineq},  for the particular case $\epsilon_k \equiv 0$,  has appeared in \cite{TangHuang2013}.

In the following we present  the well-definedness and convergence properties of the sequence $\{p^k\}$ generated by  Algorithms~\ref{Alg:InexProxae} and \ref{Alg:InexProx}. We begin with the well-definition. 
\begin{theorem}\label{th:wdef} 
Each  sequence $\{p^k\}$ generated by Algorithms~\ref{Alg:InexProxae} or \ref{Alg:InexProx}   is  well defined.
\end{theorem}
\begin{proof}
First of all note that for  each $p^{k} \in  \Omega$ and $ \lambda_k>0$,   Lemma~\ref{le:esvi} implies that 
\begin{equation} \label{eq:ieq}
0\in X(p)+N_{\Omega}(p)-\lambda_k\exp^{-1}_pp^{k}, 
\end{equation} 
has an  unique solution in $\Omega$. Since  $\mbox{dom} \,X=M$, Proposition~\ref{prop.elem.ii} and item (i) of Proposition~\ref{prop.boun.boun.} imply that $ X(p)\subseteq X^{\epsilon}(p)$ for all $p\in M$ and $\epsilon\geq 0$.  Therefore, letting  $e^{k+1}=0$ and $p^{k+1}$ as  the  solution of \eqref{eq:ieq}, we conclude they also   satisfy  \eqref{eq:ineqae}-\eqref{eq;ec}, which proof the well definition. 
\end{proof}
\begin{remark} \label{r:pcef}
 Using Proposition~\ref{prop.elem.ii} we conclude that $N_{\Omega}\subset N_{\Omega}^0$. Thus, from second. part   of Proposition~\ref{prop.elem.X}, we have  $X^{\epsilon_k}+N_{\Omega} \subset (X+N_{\Omega})^{\epsilon_k}$, for all  $k=0,1, \ldots$. Therefore, using \eqref{eq:ineq}, the following inequality holds
 \begin{equation} \label{eq:icteef}
e^{k+1}\in (X+N_{\Omega})^{\epsilon_k}(p^{k+1})- \lambda_k \exp^{-1}_{p^{k+1}}p^{k}, \qquad k=0,1, \ldots.
\end{equation}
Note that the  condition \eqref{eq:icteef}    is less restrictive than  \eqref{eq:ineqae}  and \eqref{eq:ineq}. 
\end{remark}

{\it  From now on, unless explicitly stated,   $\{p^k\}$ denotes the sequence   generated by Algorithm~\ref{Alg:InexProxae} or \ref{Alg:InexProx}}. 
 It is worth  to  noting that   $p^{k}=p^{k+1}$  implies   $p^{k+1}\in S(X,\,\Omega)$. Thus,  without loss of generality,   we  assume   that  $\{p^k\}$  is infinite.
%%%%%%%%%%%%%%%%%%%%%%%%%%%%%%%%%%%%%%%%%%%%%%
\subsection{Convergence analysis}
In this section our aim  is to prove the convergence of the sequence   $\{p^k\}$  to a point in  $S(X,\,\Omega)$. For that we first need  some auxiliary results. We begin establishing  a useful inequality. 
\begin{lemma} \label{le:pc} 
For all   $\eta>0$, the following inequality holds
$$
\left[1-\frac{\eta}{\lambda_k} \right]d^2(q, p^{k+1})\leq d^2(q, p^{k}) -d^2(p^{k}, p^{k+1}) + \frac{1}{\eta \lambda_k} \|e^{k+1}||^2 + \frac{2}{\lambda_k} \epsilon_ {k}, \qquad k=0,1, \ldots.
$$
\end{lemma}
\begin{proof}
First  we note that   \eqref{eq:ineq} or \eqref{eq:ineqae} is equivalent to 
\begin{equation} \label{eq:icte}
e^{k+1}+ \lambda_k \exp^{-1}_{p^{k+1}}p^{k}\in (X+N_{\Omega})^{\epsilon_k}(p^{k+1}), \qquad k=0,1, \ldots.
\end{equation}
Considering that $P_{qp^{k+1}}^{-1} \exp^{-1}_{q}p^{k+1}=-\exp^{-1}_{p^{k+1}}q$ and the parallel transport being isometric, the  last inclusion together  with  Definition~\ref{def.enl.X} yields
$$
-\left\langle e^{k+1}+\lambda_k \exp^{-1}_{p^{k+1}}p^{k},~\exp^{-1}_{p^{k+1}}q\right\rangle+\left\langle v,~-\exp^{-1}_qp^{k+1}\right\rangle  \geq  -\epsilon_k,\quad  
$$
for all $q\in\Omega$,   $v\in(X+N_{\Omega})(q)$ and all $k=0,1, \ldots$. In particular, if $q\in S(X,\,\Omega)$, then $0\in (X+N_{\Omega})(q)$ and the last inequality becomes
\begin{eqnarray*}
-\left\langle e^{k+1}+ \lambda_k \exp^{-1}_{p^{k+1}}p^{k}, \exp^{-1}_{p^{k+1}}q\right\rangle\geq -\epsilon_k, 
\end{eqnarray*}
for all $q\in S(X,\,\Omega)$ and all  $k=0,1, \ldots$. Using  the last inequality and \eqref{eq:coslaw} with $p_1=p^{k}$, $p_2=q$, and $p_3=p^{k+1}$, along with some  algebraic calculations,   we obtain
$$
\frac{2}{\lambda_k}\left(\left\langle e^{k+1}, \exp^{-1}_{p^{k+1}}q\right\rangle -\epsilon_ k\right)  \leq d^2(q, p^{k})-d^2(p^{k}, p^{k+1})-d^2(q,p^{k+1}), 
$$
for all $q\in S(X,\,\Omega)$ and all  $k=0,1, \ldots$. The last inequality gives
$$
d^2(q, p^{k+1})\leq d^2(q, p^{k}) -d^2(p^{k}, p^{k+1}) - \frac{2}{\lambda_k}\left\langle e^{k+1} , \exp^{-1}_{p^{k+1}}q\right\rangle + \frac{2\epsilon_k}{\lambda_k} $$
for all $q\in S(X,\,\Omega)$ and all  $k=0,1, \ldots$.  On the other hand, some algebraic manipulations yields 
$$
-\left\langle e^{k+1},\exp^{-1}_{p^{k+1}}q \right\rangle  \leq \frac{1}{2\eta}\|e^{k+1}||^2+  \frac{1}{2}\eta d^2(p^{k+1}, q). 
$$ 
Therefore, combining  two previous   last inequalities  yields the inequality of the lemma. 
\end{proof}
\begin{corollary}\label{cor:ia1}
Let $\{p^k\}$ be generated by Algorithms~\ref{Alg:InexProxae}. Then,  there exists a ${\bar k}\in \mathbb{N}$  such that, for all $ k\geq {\bar k}$, there holds
$$
 \displaystyle d^2(q, p^{k+1})\leq \left(1+\frac{2\theta_k}{\hat \lambda } \right)d^2(q, p^{k}) -  d^2(p^{k+1}, p^{k})+    \frac{2}{\hat \lambda } \left( {\theta_k} +   2 \epsilon_ k\right).
$$
\end{corollary}
\begin{proof}
First, applying Lemma~\ref{le:pc} with  $\eta= \theta_k$ and then using  \eqref{eq;ecae}   yields
$$
\left(1-\frac{ \theta_k}{\lambda_k} \right)d^2(q, p^{k+1})\leq d^2(q, p^{k}) -d^2(p^{k}, p^{k+1}) + \frac{\theta_k }{ \lambda_k}  + \frac{2}{\lambda_k} \epsilon_ {k}, \qquad k=0,1, \ldots.
$$
 It follows from  \eqref{eq:es}  that there exists  a  ${\bar k}\in \mathbb{N}$  such that $0\leq  \theta_k<  \lambda_k/2$, for   all $k\geq {\bar k}$. Thus, we conclude from the last inequality that 
\begin{equation} \label{eq.fejercat}
d^2(q, p^{k+1})\leq \left(1+\frac{ \frac{ \theta_k}{\lambda_k}}{1-\frac{ \theta_k}{\lambda_k}} \right)d^2(q, p^{k}) - \frac{\lambda_k}{\lambda_k- \theta_k} d^2(p^{k+1}, p^{k})+    \frac{1}{ \lambda_k-\theta_k} \left( {\theta_k} +   2 \epsilon_ k\right),  \quad  k\geq {\bar k}, 
\end{equation}
and  the deride inequality follows by using again   $0\leq  \theta_k<  \lambda_k/2$ and first inequality in \eqref{eq:es}. 
\end{proof}
\begin{corollary} \label{cor:ia2}
Let $\{p^k\}$ be generated by Algorithms~\ref{Alg:InexProx}. Then,  there exists a ${\bar k}\in \mathbb{N}$  such that, for all $ k\geq {\bar k}$, there holds
$$
\displaystyle d^2(q, p^{k+1})\leq \left(1+\frac{2\sigma_k}{\hat \lambda } \right)d^2(q, p^{k}) - d^2(p^{k+1}, p^{k})+   \frac{4}{\hat \lambda} \epsilon_ k.
$$
\end{corollary}
\begin{proof}
 Applying Lemma~\ref{le:pc} with  $\eta= \sigma_k$  and using \eqref{eq;ec},  we conclude that 
\begin{equation} \label{eq:eus}
\left(1-\frac{\sigma_k}{\lambda_k}\right)d^2(q, p^{k+1})\leq d^2(q, p^{k}) - \left(1-\frac{\sigma_k}{\lambda_k}\right)d^2(p^{k+1}, p^{k})+   \frac{2}{\lambda_k} \epsilon_ k, \qquad k=0,1, \ldots.
\end{equation}
On the other hand,   the forth   inequality in  \eqref{eq:es}  implies that there exists  a  ${\bar k}\in \mathbb{N}$  such that $0<  \sigma_k<  \lambda_k/2$, for   all $k\geq {\bar k}$. Hence,  using \eqref{eq:eus} together  with the first inequality in \eqref{eq:es},   we obtain   the  desired inequality.  \end{proof}
\begin{proposition} \label{pr:pc}
Let $\{p^k\}$ be a sequence  generated by Algorithms~\ref{Alg:InexProxae} or \ref{Alg:InexProx}. Then, the following statement hold:
 \begin{enumerate}
	\item [$(a)$] The sequence $\{d(p^k, q)\}$ converges, for all $q\in S(X,\,\Omega)$;
	\item [$(b)$] The sequence $\{p^k\}$ is bounded;
	\item [$(c)$] $\lim_{k\to \infty} d(p^{k+1}, p^{k})=0.$
\end{enumerate}
\end{proposition}
\begin{proof}  
If $\{p^k\}$ is   generated by Algorithm ~\ref{Alg:InexProxae},  then Corollary~\ref{cor:ia1} implies that 
$$
 \displaystyle d^2(q, p^{k+1})\leq \left(1+\frac{2\theta_k}{\hat \lambda } \right)d^2(q, p^{k})+    \frac{2}{\hat \lambda } \left[ {\theta_k} +   2 \epsilon_ k\right].
$$
Hence, item~$(a)$ follows by   applying   Lemma~\ref{le:qfm}  with $\gamma_k=2\theta_k/{\hat \lambda }$,  $\beta_k=2[ {\theta_k} +   2 \epsilon_ k]{\hat \lambda}$, $\zeta_{k}=d^2(q, p^{k})$ and $\zeta_{k+1}=d^2(q, p^{k+1})$. On the other hand,  if $\{p^k\}$ is  a sequence   generated by Algorithm~\ref{Alg:InexProx}, then  Corollary~\ref{cor:ia2} gives 
$$
\displaystyle d^2(q, p^{k+1})\leq \left(1+\frac{2\sigma_k}{\hat \lambda } \right)d^2(q, p^{k}) +   \frac{4}{\hat \lambda} \epsilon_ k.
$$
Thus, item~$(a)$ follows by  applying  Lemma~\ref{le:qfm}  with $\gamma_k=2\sigma_k/{\hat \lambda }$, $\beta_k=4\epsilon_k/{\hat \lambda }$, $\zeta_{k}=d^2(q, p^{k})$ and $\zeta_{k+1}=d^2(q, p^{k+1})$.  The item~$(b)$ is an immediate  consequence of  item~$(a)$.  The next task is to prove item~$(c)$.  If $\{p^k\}$ is   generated by Algorithm~\ref{Alg:InexProxae}, then   using  again  Corollary~\ref{cor:ia1} we have 
$$
 \displaystyle d^2(p^{k+1}, p^{k}) \leq   d^2(q, p^{k})-  d^2(q, p^{k+1}) +  \frac{2\theta_k}{\hat \lambda}d^2(q, p^{k})   + \frac{2}{\hat \lambda } \left[ {\theta_k} +   2 \epsilon_ k\right].
$$
Now, note that   \eqref{eq:es} implies   $\lim_{k\to \infty} \sigma_ k=0$ and $\lim_{k\to \infty} \epsilon_ k=0$. Therefore,   item $(a)$  together with the last inequality  imply  item $(c)$.  If $\{p^k\}$ is   generated by Algorithm~\ref{Alg:InexProx}, then it foolows from Corollary~\ref{cor:ia2} that 
$$
d^2(p^{k+1}, p^{k}) \leq  d^2(q, p^{k}) - d^2(q, p^{k+1}) + \frac{2\sigma_k}{\hat \lambda }  d^2(q, p^{k})+   \frac{4}{\hat \lambda} \epsilon_ k.
$$
On the other hand, using \eqref{eq:es}  we have  $\lim_{k\to \infty} \sigma_ k=0$ and $\lim_{k\to \infty} \epsilon_ k=0$. Therefore,   item $(a)$  together with the late  inequality  imply  item $(c)$. 
\end{proof}
\begin{theorem}\label{conv.alg.ii} 
Let $\{p^k\}$ be generated by Algorithms~\ref{Alg:InexProxae} or \ref{Alg:InexProx}. Then,  $\{p^k\}$  converges to a point   $p^*\in S(X,\,\Omega)$.
\end{theorem}
\begin{proof}
Since $\{p^k\}\subset \Omega$ and $\Omega$ is closed, item $(b)$ of  Proposition~\ref{pr:pc} implies that there exists ${\bar p}\in \Omega$ a cluster point of $\{p^{k}\}$. Let    $\{p^{k_j}\}$ be  a subsequence of $\{p^k\}$ such that $\lim_{j\to \infty}p^{k_j}={\bar p}$.  Our first aim is to prove that ${\bar p}\in S(X,\,\Omega)$.  For that,   using   inclusion \eqref{eq:ineq} or \eqref{eq:ineqae}, there exist $u^{k_j+1}\in X^{\epsilon_{k_j}}(p^{k_j+1})$  such that  
\begin{equation} \label{eq:isss}
e^{k_j+1}+ \lambda_{k_j} \exp^{-1}_{p^{k_j+1}}p^{k_j}- u^{k_j+1}   \in N_{\Omega}(p^{k_j+1}), \qquad j=0,1, \ldots.
\end{equation}
On the other hand,    item $(c)$ of  Proposition~\ref{pr:pc} implies  that  $\lim_{j\to \infty}p^{k_j+1}={\bar p}$. Moreover,  considering that $\{\theta_k\}$ and  $\{\sigma_ k\}$ are  bounded, it follows from  \eqref{eq;ecae}, respectively \eqref{eq;ec},  and   item (c) of Proposition~\ref{pr:pc} that $\lim_{k\to \infty} e^k=0$.  Letting  ${\bar \epsilon}=\sup_k\epsilon_k$, the first part   of  Proposition~\ref{prop.elem.X}  implies that  $u^{k_j+1}\in X^{\epsilon_{k_j}}(p^{k_j+1})\subset X^{\bar \epsilon}(p^{k_j+1})$, for all  $j=0,1, \ldots$. Thus,  considering that $\{p^k\}$ is bounded, we conclude from   Proposition~\ref{prop.boun.boun.} that $\{u^{k_j+1}\}$ is also bounded. Without loss of generality we assume that $\lim_{j\to \infty}u^{k_j+1}={\bar u}$. Hence, taking into account that  $\lim_{k\to \infty} \epsilon_ k=0$, $\lim_{j\to \infty}p^{k_j+1}={\bar p}$ and  $u^{k_j+1}\in X^{\epsilon_{k_j}}(p^{k_j+1})$, for all  $j=0,1, \ldots$,    it follows from  Proposition~\ref{prop.elem.ii} and  Proposition~\ref{prop.conv.alg.} that  ${\bar u}\in X^{0}({\bar p})= X({\bar p})$. Therefore,  taking limit in  \eqref{eq:isss} and considering Proposition~\ref{eq:nccc} we conclude that 
$-{\bar u}   \in N_{\Omega}({\bar p})$. Due to ${\bar u}\in  X({\bar p})$  we have $0\in X({\bar p}) + N_{\Omega}({\bar p})$, which implies that  $\bar{p}\in S(X,\,\Omega)$. Moreover, using  item $(a)$ of  Proposition~\ref{pr:pc} que obtain  that  the sequence $\{d(p^k, {\bar p})\}$ converges.  Considering that $\lim_{j\to \infty}p^{k_j}={\bar p}$, we have   $\lim_{k\to \infty} d(p^{k_j}, {\bar p})=0$.  Therefore, we conclude that  $\lim_{k\to \infty} d(p^{k}, {\bar p})=0$, or equivalently, $\lim_{k\to \infty} p^{k}= {\bar p}$, which concludes the proof. 
 \end{proof}
\begin{remark}
In \cite{WangLiLopezYao2015, WangLiChongGenaro2016} is  presented an inexact version of the proximal point method for to find singularity of a vector field  on Hadamard manifolds. These papers differs from the present paper in two ways, namely,  \cite{WangLiLopezYao2015, WangLiChongGenaro2016}  use only  absolute summable error criteria and the enlargement  $X^{\epsilon} $ of  $X$ was not considered.  It is worth noting that,  the enlargement  $X^{\epsilon} $  is an (outer) approximation to $X$.  Consequently,    even in the linear setting,   the  proximal subproblem using the  enlargement    has the advantage of providing more latitude and more robustness to the methods used for solving it; see   \cite{BurachikIusemSvaiter1997, BurachikIusem2008}.
\end{remark}
%%%%%%%%%%%%%%%%%%%%%%%%%%%%%%%%%%%%%%%%%%%%%%
\section{Applications} \label{sec:appl}
The general Problem~\eqref{eq.vip} has as particular instances  the optimization problem,  equilibrium problem and nonlinear optimization problem. The aim of this section is  to apply the results obtained in the previous section to these  particular instances.  For  each problem studied,   a version of the  Algorithm~\ref{Alg:InexProx} is  stated  to  solve it. Since  a  version of the Algorithm~\ref{Alg:InexProxae}   can be stated following the same idea, it  will  be omitted.
%%%%%%%%%%%%%%%%%%%%%%%%%%%%%%%%%%%%%%
\subsection{ Inexact proximal point method for optimization} \label{sec4}
%%%%%%%%%%%%%%%%%%%%%%%%%%%%%%%%%%%%%%
In this section, we apply the results of the previous section   to  obtain an inexact proximal  point method for the constrained optimization problems in Hadamard manifolds.  Given a closed and convex set $\Omega \subset M$ and a  convex function $f:M \rightarrow \mathbb{R}$,  the {\it constrained optimization problem}  consists of
\begin{equation}  \label{eq.cop}
 \min  ~ f(p) , \qquad  ~ p\in\Omega.
\end{equation}
 The  problem in \eqref{eq.cop} is equivalently stated as follows 
\begin{eqnarray} \label{eq:eqpr}
 \min  ~ (f+\delta_{\Omega})(p), \qquad   ~ p\in M.
\end{eqnarray}
where $\delta_\Omega$ is the  indicate functionof $\Omega$. Hereafter,    $S(f,\Omega)$ denotes    the solution  set of  the problem in \eqref{eq.cop}.  It well know that \eqref{eq:eqpr} can be stated  as the variational inequality problem~\eqref{eq.vip}.  In fact, first note that  due to convexity of  the set $\Omega$ and of the function  $f$ we conclude that   $f+\delta_{\Omega}$  is also convex. Thus, by using  Proposition~\ref{pr:pif} we have   
$$
 \partial (f+\delta_{\Omega})(p)=\partial f(p)+N_{\Omega}(p), \qquad  p\in \Omega.
$$
Therefore,    $p^*\in S(f,\Omega)$ if,  and  only if,~    $0\in \partial f(p^*)+N_{\Omega}(p^*)$. Therefore, \eqref{eq:eqpr}  is equivalent  to find an $p^*\in\Omega$ satisfying  the inclusion
\begin{equation} \label{eq:efvipf}
0\in \partial f(p)+N_{\Omega}(p).
\end{equation} 
In order to present a version of Algorithm~\ref{Alg:InexProx} to solve \eqref{eq.cop} or equivalently  \eqref{eq:efvipf}, we need to consider the   {\it enlargement of the subdifferential of} $f$,   denoted  by  $\partial^{\epsilon}f: M   \rightrightarrows  TM $, which    is defined by
$$
 \partial^{\epsilon} f(p):=\left\{ u\in T_pM~:~  \left\langle \mbox{P}_{qp}^{-1} u-v, \,  \exp_{q}^{-1}p\right\rangle \geq -\epsilon, ~ q\in M, ~  v\in \partial f(q) \right\},  \qquad \epsilon\geq 0.
$$
{\it To state the   version of Algorithm~\ref{Alg:InexProx} to solve \eqref{eq.cop} or equivalently \eqref{eq:efvipf},  take  three exogenous  sequences   of nonnegative real  numbers}, $\{\lambda_k\}$,   $\{\epsilon_k\}$ and $\{\sigma_k\}$ satisfying \eqref{eq:es}. 
Then, the  inexact proximal point method  for the optimization problem \eqref{eq:efvipf}  is introduced as follows:\\

 \vspace{.2cm}
\hrule
\begin{algorithm} \label{Alg:InexProxOp}
 {\bf Inexact proximal point method for optimization problems}  \\
\hrule
\begin{description}
\item[0.] Take $\{\lambda_k\}$,   $\{\epsilon_k\}$ and $\{\sigma_k\}$ satisfying \eqref{eq:es} and   $p^0 \in  \Omega$. Set $k=0$.
\item [1.] Given $p^{k} \in \Omega$, compute   $p^{k+1} \in \Omega$  and  $e^{k+1}\in T_{p^{k+1}}M$ such that
\begin{equation} \label{eq.pia.iif}
e^{k+1} \in \partial^{\epsilon_k}f(p^{k+1})+N_{\Omega}(p^{k+1})-2\lambda_k\exp^{-1}_{p^{k+1}}x^k, 
\end{equation} 
\begin{equation} \label{eq;eciif}
\|e^{k+1}\| \leq \sigma_ {k}d(p^{k}, p^{k+1}), 
\end{equation}
\item[ 2.] If  $p^{k}=p^{k+1}$, then {\bf stop}; otherwise,   set $k\gets k+1$, and go to step~{\bf  1}.
\end{description}
\hrule
\end{algorithm}
\vspace{.2cm}

\begin{remark}
In case, $\epsilon_k\equiv 0$, $e^{k+1}\equiv 0$ and  $\Omega=M$,   the Algorithm~\ref{Alg:InexProxOp} generalize  the algorithm proposed by  Ferreira and Oliveira~\cite{FerreiraOliveira2002},   and  the method (5.15) of  Chong Li et. al. \cite{LiLopesMartin-Marquez2009}. For  $\epsilon_k=0$, inexact variations of \eqref{eq.pia.iif} with absolute erros  can be found in \cite{WangLiLopezYao2015} and \cite{TangHuang2013} for relative erro. Finally, letting    $e^{k+1}\equiv 0$, the Algorithm~\ref{Alg:InexProxOp} retrieves   the one presented in \cite{BatistaGlaydstonFerreira2016}. 
\end{remark}
In the following we state  a convergence result for the sequence generated by  \eqref{eq.pia.iif} and \eqref{eq;eciif}. First  note that, for  considering that   $\mbox{dom}f=M$,   Theorem~\ref{mmsub} implies that  $\partial f$ is maximal monotone. Hence, from Proposition~\ref{pr:pif}  we have   $N_{\Omega}= \partial \delta_{\Omega}$. Therefore,     by applying Theorems~\ref{th:wdef}  and \ref{conv.alg.ii} with $X=\partial f$ we obtain  the following theorem.
\begin{theorem} \label{conv.alg.}
Assume that $S(f,\,\Omega)\neq \varnothing$. Then, the sequence $\{p^k\}$ generated by \eqref{eq.pia.iif} and \eqref{eq;eciif}   is well defined and  converges to a point   $p^*\in S(f,\,\Omega)$.
\end{theorem}

%\textcolor{red}{Ate aqui}

In the next remark, we highlight the advantage of using the enlargement of the subdifferential  instead of the $\epsilon$-subdifferential of $f$.

\begin{remark}
The {\it $\epsilon$-subdifferential} of $f$, denoted by  $\partial_{\epsilon}f : M   \rightrightarrows  TM$, is given by
$$
\partial_{\epsilon}f(p):=\left\{u\in T_pM ~ : ~ f(q) \geq f(p)+\left\langle u,  \exp^{-1}_pq\right\rangle - \epsilon, ~  q\in M\right\}, \qquad \epsilon\geq 0.
$$
As can be seen in \cite{BatistaGlaydstonFerreira2016}, the enlargement of the subdifferential of $f$  is bigger than its $\epsilon$-subdifferential, i.e., for each  $p\in M$, there holds $\partial_\epsilon f(p)\subseteq \partial^\epsilon f(p)$. 
Taking into account that this inclusion may be strict, we can to state that the iteration in \eqref{eq.pia.iif} using the enlargement of $\partial f(\cdot)$ has the advantages of providing more latitude and more robustness than a method using the $\epsilon$-{subdifferential} of $f$.
\end{remark}
%%%%%%%%%%%%%%%%%%%%%%%%%%%%%%%%%%%%%%%%%%%%%%
\subsection{Inexact proximal point method for equilibrium problems} \label{Sec:AppEP}
In this section, by using  the results of Section~\ref{Sec:InProxVIP},  we present  a version of the  inexact proximal point method for equilibrium problems  in Hadamard manifolds. For that we need some preliminaries. Let $C\subset M$ be a nonempty, closed and convex set and $F:M \times M \to \R$ be a bifunction satisfying  the following standard assumptions:
\begin{enumerate} 
\item [{\bf H1.}] $F(\cdot,y): M \rightarrow \R$ is upper semicontinuous for all $y \in M;$
\item [{\bf H2.}] $F(x,\cdot): M \rightarrow \R$ is  convex, for all $x \in M.$;
\item [{\bf H3.}] $F$ is  monotone on ${C}$, i.e.,  $ F(x,y)+F(y,x)\leq 0$,  for all $x,y \in {C}$;
\item [{\bf H4.}]  $F(x,x)=0,$ for all $x\in {C}.$
\end{enumerate}
The {\it equilibrium problem}  associated to the  set $C$ and the bifunction $F$, denoted by $\mbox{EP}(F,{C})$, is stated as follows:  Find $x^* \in {C} $ such that 
\begin{equation}\label{eq:epave}
F(x^*,y)\geq 0,\qquad  y\in {C}.
\end{equation}
Denote by $S(F, {C})$ the solution set of the $\mbox{EP}(F,{C})$. By using \cite[Proposition 3.3]{Li2019} we obtain that    \eqref{eq:epave}  is equivalent  to find an $p^*\in\Omega$ satisfying  the inclusion
\begin{equation} \label{eq:epvii}
0\in \partial_{2} F(p, \cdot)(p)+N_{\Omega}(p), 
\end{equation}
where $\partial_{2}F(p, \cdot)$ denotes the subdifferential of $F$ with respect to  the second argument. 
We also assume that 
\begin{enumerate} 
\item [{\bf H5.}] The set $S(F, C)$ is nonempty.
\end{enumerate}
\begin{remark}
Assumptions {\bf H1}-{\bf H4} are standard for the study of equilibrium problems in linear spaces, see \cite{iusem20031,iusem2009,Iusem2010}. It is worth to notting that assumption {\bf H5} can be reached under suitable condition on the set $C$ or the bifunction $F$; papers addressing this issue include, but are not limited to,   \cite{Batista2015, konnov2005, CLMM2012,Li2019,  ZhouHuang2019}. To the best of our knowledge, our approach brings a first proposal of an inexact proximal method for equilibrium problems. It is worth noting that an exact version has been first introduced in \cite{CLMM2012} and, by using variational inequality theory, reaffirmed for genuine Hadamard manifolds in \cite{Li2019}.
\end{remark}

Before presenting  a version of Algorithm~\ref{Alg:InexProx} to solve \eqref{eq:epave}, or equivalently \eqref{eq:epvii},  we  need to consider the   {\it enlargement of the subdifferential of} $F$    with respect to  the second argument, denoted  by  $\partial_{2}^{\epsilon}F(z, \cdot): M   \rightrightarrows  TM $, for each fixed $z \in {C}$, which    is introduced  as follows 
\begin{equation} \label{eq:edsa}
\partial_{2}^{\epsilon}F(z,x):=\left\{w \in T_{p}M:~ F(z,y)\geq F(z,x)+ \left \langle w, \exp^{-1}_x y \right \rangle-\epsilon, ~y \in M\right\}.
\end{equation} 
{\it To state the   version of Algorithm~\ref{Alg:InexProx} to solve the equilibrium problem~\eqref{eq:epave} or equivalently  \eqref{eq:epvii},  take  three exogenous  sequences   of nonnegative real  numbers}, $\{\lambda_k\}$,   $\{\epsilon_k\}$ and $\{\sigma_k\}$ satisfying \eqref{eq:es}. In case, the inexact proximal point method  for solving  \eqref{eq:epave} is introduced as follows:\\

\vspace{.2 cm}
\hrule
\begin{algorithm} \label{Alg:ProxEP}
 {\bf Inexact proximal point method for equilibrium problems}  \\
\hrule
\begin{description}
\item[0.] Take $\{\lambda_k\}$,   $\{\epsilon_k\}$ and $\{\sigma_k\}$ satisfying \eqref{eq:es},    $x^0 \in  {C}$ and $\sigma>1$. Set $k=1$.
\item [1.] Given $x^{k} \in {C}$, compute   $x^{k+1} \in {C}$  and  $e^k \in T_{x^k}M$ such that
\begin{equation} \label{eqn03}
 e^{k+1} \in \partial_{2}^{\epsilon_k}F(x^{k+1},x^{k+1}) + N_{\Omega}(x^{k+1}) -\lambda_k  {\rm exp}_{x^{k+1}}^{-1}x^{k},
 \end{equation}
\begin{equation} \label{eqn05}
\|e^{k+1}\| \leq \sigma_ {k}d(p^{k}, p^{k+1}).
\end{equation}
\item[ 2.] If  $x^{k-1}=x^k$ or $x^k\in  $, then {\bf stop}; otherwise,   set $k\gets k+1$, and go to step~{\bf  1}.
\end{description}
\hrule
\end{algorithm}
\vspace{.2 cm}

\begin{remark}
The  Algorithm~\ref{Alg:ProxEP}  can be seen as an inexact version  of  the   following  iterative scheme  considered in \cite{CLMM2012}: For $x^{k} \in {C}$,  compute $x^{k+1} \in {C}$ such that 
\begin{equation} \label{eq1}
F(x^{k+1},x)-\lambda_k \la {\rm exp}_{x^{k+1}}^{-1}x^{k}, {\rm exp}_{x^{k+1}}^{-1}x  \ra\geq 0, \qquad   x\in {C}. 
\end{equation}
Indeed,    given   $x^{k}$ and  $x^{k+1} \in {C} $  satisfying \eqref{eq1} we have  
\begin{equation} \label{eq1ax}
F(x^{k+1},x)+ \delta_{C}(x) -\left(F(x^{k+1},x^{k+1})+ \delta_{C}(x^{k+1}) + \left\la \lambda_k {\rm exp}_{x^{k+1}}^{-1}x^{k}, {\rm exp}_{x^{k+1}}^{-1}x  \right\ra \right) \geq 0,  
\end{equation}
for all $ x\in M$.  Since the function $p\mapsto (F(p^{k+1},\cdot) +  \delta_{C}(\cdot))(p)$  is convex, it follows from the definition of the subdifferential  
 that $\lambda_k {\rm exp}_{x^{k+}}^{-1}x^{k} \in  \partial_2 (F(x^{k+1},\cdot) +  \delta_{C}(\cdot) ) (x^{k+1})$. Hence, by using Proposition~\ref{pr:pif} we obtain 
$$
  0\in \partial_2 F(x^{k+1},x^{k+1})- \lambda_k {\rm exp}_{x^{k+1}}^{-1}x^{k} +  N_{\Omega}(x^{k+1}), 
$$
which implies that    $x^{k}$  and   $x^{k+1} $ also  satisfy \eqref{eqn03} and \eqref{eqn05} with $e^{k+1}=0$ and  $\epsilon_k=0$.
\end{remark}
In the following we state  a convergence result for the sequence generated by  \eqref{eqn03} and \eqref{eqn05}. First  note that, for  considering that   $\mbox{dom}F(p, \cdot)=M$,     Theorem~\ref{mmsub} implies that  $\partial_{2} F(p, \cdot)$ is maximal monotone, for all $p\in M$. Moreover, Proposition~\ref{pr:pif}  implies that    $N_{\Omega}= \partial \delta_{\Omega}$. Therefore,     by applying Theorems~\ref{th:wdef}  and \ref{conv.alg.ii} with $X=\partial_{2} F(p, \cdot)$ we obtain  the following theorem.
\begin{theorem} \label{conv.alg.}
The sequence $\{p^k\}$ generated by  \eqref{eqn03} and \eqref{eqn05}   is well defined and  converges to a point   $p^*\in S(f,\,\Omega)$.
\end{theorem}
\begin{proof}
First  note that, for  considering that   $\mbox{dom}F(p, \cdot)=M$,     Theorem~\ref{mmsub} implies that  $\partial_{2} F(p, \cdot)$ is maximal monotone, for all $p\in M$. Moreover, Proposition~\ref{pr:pif}  implies that    $N_{\Omega}= \partial \delta_{\Omega}$. Therefore,     by applying Theorems~\ref{th:wdef}  and \ref{conv.alg.ii} with $X=\partial_{2} F(p, \cdot)$ we obtain  the following theorem.
\end{proof}
%%%%%%%%%%%%%%%%%%%%%%%%%%%%%%%%%%%%%%%%%%%%%%%%%%%%%%
\subsection{ Inexact proximal point method for nonlinear optimization problem} \label{sec4}
%%%%%%%%%%%%%%%%%%%%%%%%%%%%%%%%%%%%%%
In this section, we apply the results of the previous section to  obtain an inexact proximal  point method for the  nonlinear optimization problem  in the form
\begin{equation}  \label{eq:nlp}
 \min  ~ f(p) , \qquad  \quad ~ p\in \left\{p\in M: ~g(p)\leq 0, ~h(p)=0\right\}, 
\end{equation}
where $M$ is a Hadamard manifol,  the objective function  $f:M \rightarrow \mathbb{R}$  and the constraint functions   $g=(g_1, \ldots, g_m):M \rightarrow \mathbb{R}^m$ and $h=(h_1, \ldots, h_\ell):M \rightarrow \mathbb{R}^\ell$ are assumed to be continuously  differentiable and convex. In order to state the problem \eqref{eq:nlp} as  the variational inequality problem in   \eqref{eq.vip}, we first recall the  first-order necessary optimality conditions  in Karush-Kuhn-Tucker (KKT) form. It is worth noting that recently the KKT conditions were addressed in \cite{BergmannHerzog2019}. Let  ${\cal L}: M\times \mathbb{R}_{+}^m \times \mathbb{R}^\ell \to \mathbb{R}$  be the   Lagrangian associated with \eqref{eq:nlp} defined by 
\begin{equation}  \label{eq:Lagrang}
{\cal L} (p, \mu, \lambda):=f(p) +\sum_{i=1}^{m}\mu_ig_i(p)+ \sum_{j=1}^{\ell}\lambda_jh_j(p).
\end{equation}
Since the functions $f$,   $g$  and $h$ are  continuously  differentiable, it follows from \eqref{eq:Lagrang} that  KKT conditions are given by 
\begin{align}  
\grad_ p{\cal L} (p, \mu, \lambda):=\grad f(p) +\sum_{i=1}^{m}\mu_i\grad g_i(p)+ \sum_{j=1}^{\ell}\lambda_j\grad h_j(p)&=0\label{eq:kkk1}\\
                                                                                                                                                                                     g_i(p)&\leq0, \quad i=1, \ldots m \label{eq:kkk2}\\                                                                                                                                                                                                                                                                                                                                                                                                                                                                                                                                                                                                                                                                                                                                                                                                                                                                                                                  
                                                                                                                                                                                     h_j(p)&=0, \quad j=1, \ldots \ell  \label{eq:kkk3}\\
                                                                                                                                                                                    \mu_i g_i(p)&=0, \quad i=1, \ldots \label{eq:kkk4}\\
                                                                                                                                                                                     \mu_i &\geq 0, \quad i=1, \ldots m \label{eq:kkk4}
\end{align}
Let ${\widetilde M}:=M\times \mathbb{R}^m \times \mathbb{R}^\ell$ be  the product manifold with the induced  product metric, for more details see \cite{Sakai1996}.  Then,  the tangent plane at ${\tilde p}:=(p, \mu, \lambda)\in {\widetilde M}$  is  $T_{\tilde p}{\widetilde M}:=T_{p}M\times \mathbb{R}^m\times \mathbb{R}^\ell$ and  the exponential map ${\widetilde{\rm exp}}_{\tilde p}: T_{\tilde p}{\widetilde M} \to {\widetilde M}$  is given by 
$$
{\widetilde{\rm exp}}_{\tilde p}{\tilde w}:= \left(\exp_{p}{w}, ~ \mu+u,~ \lambda+v\right), \qquad \quad {\tilde w}:=(w, u, v)\in T_{{\tilde p}}{\widetilde M}.
$$
where $\exp_{p}$ is the exponential map of $M$ a $p\in M$. Consequently, the inverse of ${\widetilde{\rm exp}}_{\tilde p}$  is given by 
$$
{\widetilde{\rm exp}}_{\tilde p}^{-1}{\tilde q}:= \left(\exp_{p}^{-1}{q}, ~ {\nu}-\mu,~  {\zeta}-\lambda\right), \qquad \quad {\tilde q}:=({q}, { \nu}, {\zeta})\in {\widetilde M}.
$$
Let    ${\tilde \Omega}:=M\times \mathbb{R}_{+}^m\times \mathbb{R}^\ell\subset {\widetilde M}$, which is  convex set in $ {\widetilde M}$. In this case,  the   {\it normal cone} of the set  ${\tilde \Omega}$ at a point ${\tilde p}\in \Omega$  is given by
$$
N_{\tilde \Omega}({\tilde p}):=\left\{ {\tilde w}\in T_{\tilde p}{\widetilde M}:~\left\langle{\tilde w}, {\widetilde{\rm exp}}_{\tilde p}^{-1}{\tilde q} \right\rangle\leq 0,  ~ {\tilde q}\in \Omega \right\}.
$$
Since  the functions $f$,   $g$  and $h$ are  continuously  differentiable and convex,  the definition \eqref{eq:Lagrang} implies that  the vector field    $X: M\times \mathbb{R}_{+}^m\times \mathbb{R}^\ell \to TM\times \mathbb{R}^m\times \mathbb{R}^\ell$ defined by
\begin{equation} \label{eq:vfala}
X({\tilde p}):=\begin{bmatrix}
\grad_ p{\cal L} ({\tilde p})\\
-g(p)\\
h(p)
\end{bmatrix} \in  T_{\tilde p}{\widetilde M}, \qquad {\tilde p}:=(p, \mu, \lambda)\in {\widetilde M}.
\end{equation} 
is maximal monotone.  Moreover, using  the definition of normal cone of the set ${\tilde \Omega}$ we conclude that \eqref{eq:kkk1}-\eqref{eq:kkk4}  is equivalent to 
\begin{equation}  \label{eq:VIPnlp}
0\in X({\tilde p}) + N_{\tilde \Omega}({\tilde p}), \qquad {\tilde p}:=(p, \mu, \lambda)\in {\widetilde M}.
\end{equation}
To  the definition of enlargement of  $X$,  consider  the parallel transport on ${\widetilde M}$ from ${\tilde p}$ to  ${\tilde q}$  as being 
$$
{\tilde P}_{{\tilde p}{\tilde q}} {\tilde w}:= \left( {P}_{pq}{w}, u, v  \right), \qquad {\tilde w}:=(w, u, v)\in T_{{\tilde p}}{\widetilde M}.
$$
where $ {P}_{pq}$ is the  parallel transport on ${ M}$ from ${p}$ to  ${ q}$. Then,  the enlargement of  vector field  $X^{\epsilon}: M   \rightrightarrows  TM $  associated to  $X$   is defined by
\begin{equation} \label{enl.X}
    X^{\epsilon}({\tilde p}):=\left\{ {\tilde w}:=(w, u, v)\in T_{{\tilde p}}{\widetilde M}:~ \left\langle {\tilde P}_{{\tilde q}{\tilde p}}^{-1}{\tilde w}- {\tilde z}, \widetilde{\rm exp}_{\tilde q}^{-1}{\tilde p} \right\rangle \geq  -\epsilon, ~  {\tilde q}\in M, ~  {\tilde z}\in X({\tilde q}) \right\}, 
\end{equation}
for all $ {\tilde p}:=(p, \mu, \lambda)\in {\widetilde M}$.  Finally,   taking   three exogenous  sequences   of nonnegative real  numbers, $\{\lambda_k\}$,   $\{\epsilon_k\}$ and $\{\sigma_k\}$ satisfying \eqref{eq:es}, the {\it  inexact proximal point method  for solving  \eqref{eq:nlp} } is introduced as follows: \\

\vspace{.2 cm}

\hrule
\begin{algorithm} \label{Alg:ProxEPNP}
 {\bf  Inexact proximal point method for nonlinear optimization problem}  \\
\hrule
\begin{description}
\item[0.] Take $\{\lambda_k\}$,   $\{\epsilon_k\}$ and $\{\sigma_k\}$ satisfying \eqref{eq:es},    ${\tilde p}^0 \in  {\tilde \Omega}$ and $\sigma>1$. Set $k=1$.
\item [1.] Given ${\tilde p}^{k} \in {\tilde \Omega}$, compute   ${\tilde p}^{k+1} \in {C}$  and  $e^k \in T_{{\tilde p}^k}M$ such that
\begin{equation} \label{eqn0np}
 {\tilde e}^{k+1} \in X^{\epsilon}({\tilde p}^{k+1}) + N_{\tilde \Omega}({\tilde p}^{k+1}) -\lambda_k  {\widetilde{\rm exp}}_{{\tilde p}^{k+1}}^{-1}{\tilde p}^{k},
 \end{equation} 
\begin{equation} \label{eqn05np}
\|{\tilde e}^{k+1}\| \leq \sigma_ {k}d({\tilde p}^{k}, {\tilde p}^{k+1}).
\end{equation}
\item[ 2.] If  ${\tilde p}^{k-1}={\tilde p}^k$ or ${\tilde p}^k\in  $, then {\bf stop}; otherwise,   set $k\gets k+1$, and go to step~{\bf  1}.
\end{description}
\hrule
\end{algorithm}
\vspace{.2 cm}

First note that  $X$ defined in \eqref{eq:vfala}  satisfies {\bf A1} and  {\bf A2}.  Moreover, under suitable constraint qualifications $X$  also satisfies {\bf A2}, see \cite[Theorem 11]{BergmannHerzog2019}. Therefore,  we can apply Theorem~\ref{conv.alg.ii} to obtain the following result. 
\begin{theorem} \label{conv.alg.}
The sequence $\{p^k\}$ generated by  \eqref{eqn0np} and \eqref{eqn05np}   is well defined and  converges to a point   $p^*\in S(f,\,\Omega)$.
\end{theorem}

%%%%%%%%%%%%%%%%%%%%%%%%%%%%%%%%%%%
\section{Conclusions} \label{secfr}
%%%%%%%%%%%%%%%%%%%%%%%%%%%%%%%%%%%
In this paper we combine the ideas in \cite{rocka01}  and \cite{BurachikIusemSvaiter1997} to introduce an  inexact  proximal point method for solving variational inequality  problems  on Hadamard manifolds.   As a proposal of future work it would be interesting to study local version  of our results on arbitrary Riemannian manifolds. Note that for this purpose, a local version of the formula \eqref{eq:coslaw}  will be required.

%%%%%%%%%%%%%%%%%%%%%%%%%%%%%%%%%%%%%%%%%%%%%%%%%%%%%%%%%%%%%%%%%%%%%%%%%%
%\begin{center}{\bf{Acknowledgments.}}
%\end{center}\noindent
%We are grateful to the anonymous referees that suggested several important inclusions and changing on the paper.

%\bibliographystyle{habbrv}
%\bibliography{ProxEnlargement}

\def\cprime{$'$} \def\cprime{$'$}

\end{document}